\documentclass[11pt,english]{article}
\usepackage[T1]{fontenc}
\usepackage[latin9]{inputenc}
\usepackage{geometry}
\usepackage{verbatim}
\usepackage{calc}
\usepackage{amsthm}
\usepackage{amsmath}
\usepackage{amssymb}
\usepackage{graphicx}
\usepackage{esint}

\makeatletter
\theoremstyle{plain}
\newtheorem{thm}{\protect\theoremname}
  \theoremstyle{definition}
  \newtheorem{defn}[thm]{\protect\definitionname}
  \theoremstyle{definition}
  \newtheorem{example}[thm]{\protect\examplename}
  \theoremstyle{remark}
  \newtheorem{rem}[thm]{\protect\remarkname}
  \theoremstyle{plain}
  \newtheorem{lem}[thm]{\protect\lemmaname}

\DeclareMathOperator{\Vol}{Vol}

\DeclareMathOperator{\hup}{{\bf h}_{upp}}
\DeclareMathOperator{\hit}{{\bf h}_{ite}}
\DeclareMathOperator{\itt}{ite}

\DeclareMathOperator{\fI}{{\cal I}}
\DeclareMathOperator{\dd}{(\textit{d}_1,\textit{d}_2,\dots,\textit{d}_\textit{n})}
\DeclareMathOperator{\Bsys}{\textit{B}_{\cal I}-system}
\DeclareMathOperator{\BsysJ}{\textit{B}_{\cal J}-system}
\DeclareMathOperator{\Ecirc}{\textit{E}_{\circ}\left({\cal I};\textit{d}_{1},\dots,\textit{d}_{\textit{n}}\right)}
\DeclareMathOperator{\Edot}{\textit{E}_{\bullet}\left({\cal I};\textit{d}_{1},\dots,\textit{d}_{\textit{n}}\right)}

\DeclareMathOperator{\BsysPrime}{\textit{B}_{\cal I'}-system}

\makeatother

\usepackage{babel}
  \providecommand{\definitionname}{Definition}
  \providecommand{\examplename}{Example}
  \providecommand{\lemmaname}{Lemma}
  \providecommand{\remarkname}{Remark}
\providecommand{\theoremname}{Theorem}

\begin{document}

\title{Combinatorial Excess Intersection}

\author{Jose Israel Rodriguez%
\thanks{The author is supported by the US National Science Foundation DMS-0943745. %
}}
\date{}

\maketitle
\begin{abstract}
We provide formulas and develop algorithms for computing the excess
numbers of an ideal. The solution for monomial ideals is given by
the mixed volumes of polytopes. These results enable us to design
numerical algebraic geometry homotopies to compute excess numbers
of any ideal. 
\end{abstract}
\section{Introduction}

Consider a homogeneous ideal $\fI\subset\mathbb{C}\left[x_{0},\dots,x_{n}\right]$, and
let $f_{1},\dots,f_{n}$ be homogenous polynomials in $\fI$. 
Since $\left(f_{1},\dots,f_{n}\right)\subset\fI$,
we have  ${\bf V}\left(f_{1},\dots,f_{n}\right)\supset{\bf V}\left(\fI\right)$.
The $\emph{excess intersection }$ of the variety of $\left(f_{1},f_{2},\dots,f_{n}\right)$
with respect to the variety of $\fI$ is defined as the quasiprojective variety
${\bf V}\left(f_{1},\dots,f_{n}\right)\backslash{\bf V}\left(\fI\right)$.
We define the $\emph{excess number }$ $E_{\bullet}\left(\fI;f_{1},\dots,f_{n}\right)$
of an ideal $\fI$ to be the number of  solutions in
${\bf V}\left(f_{1},\dots,f_{n}\right)\backslash{\bf V}\left(\fI\right)$.

Excess intersections are a well  studied problem with applications in enumerative geometry, machine learning  \cite{franzConj,franzGeneric}, and algebraic statistics \cite{jostEC}. 
In addition, there is a well developed theory of Segre classes to study this problem that has been exploited in 
 \cite{segreProj,segre2011,segreToric} using computational algebraic geometry as well.
Recent work by Paolo Aluffi has pushed this area even further in \cite{paSegrePolytopes}.
However, the motivation for this paper came at the 2012 Institute for Mathematics and
its Applications Participating Institution Summer Program for Graduate
Students in Algebraic Geometry for Applications by Mike Stillman.
We will focus on the numerical algebraic geometry perspective,
where it is ideal to solve $square$ systems of equations, meaning
the number of unknowns equals the number of equations.
So by understanding
the zero-dimensional solutions of an excess intersection of an ideal, we can 
 study the ideal itself. 
Our computations were performed with ${\tt Bertini}$, ${\tt PHCpack}$,
and ${\tt Macaulay2}$.

We begin our study in the case that $\fI$ is an ideal generated by 
$B_1,B_2,\dots,B_l$,
and $f_{1},\dots,f_{n}$ define a $\Bsys$ of equations of degree $\dd$.

\begin{defn}\label{defBsys}
Let $\fI$ be an ideal of $\mathbb{C}\left[x_{0},\dots,x_{n}\right]$
generated by $B_{1},\dots,B_{l}$ whose respective degrees are
$p_{1},\dots,p_{l}$.
Suppose $\dd$ is such that $$\min\dd\geq \max(p_1,\dots,p_l).$$
Let $a_{ij}$ denote a
form of  degree  $d_{i}-\deg B_{j}$. 
If the forms $f_{1},\dots,f_{n}$ are given by 
\[
\left[\begin{array}{l}
f_{1}\\
f_{2}\\
\vdots\\
f_{n}
\end{array}\right]=\left[\begin{array}{lll}
a_{11} & \cdots & a_{1l}\\
a_{21} & \cdots & a_{2l}\\
\vdots &  & \vdots\\
a_{n1} & \cdots & a_{nl}
\end{array}\right]\left[\begin{array}{l}
B_{1}\\
\vdots\\
B_{l}
\end{array}\right],
\]
then 
we say $f_{1},\dots,f_{n}$ are 
a $\Bsys$ of degree $\dd$.
\end{defn}

The space of $\Bsys$'s with degree $\dd$ is parameterized by the coefficients of the homogeneous polynomials $a_{ij}$. 
If $\fI$  is generated by $B_1,\dots,B_l$, then we denote the excess number of a general $\Bsys$  with degree $\dd$ as $\Edot$. 

In the first section we will be interested in determining excess numbers of $\Bsys$'s where $B_1,\dots,B_l$  are monomials. 

At times it will be more convenient to work with the equivalence number 
$$\Ecirc:=d_{1}\cdots d_{n}-\Edot.$$
This definition is inspired by the notion of the  \emph{equivalence of
an ideal} in \cite{intersectionBook} {[}Chapter $6${]}. This number is the difference
between the Bezout bound and the excess number in the cases we consider. 
The contributions
of the paper include numerical algebraic geometry algorithms to compute
excess numbers and a combinatorial proof of the theorem below.
This theorem can be proven easily using Fulton-MacPherson intersection theory, and in fact doing so generalizes the result to any ideal generated by a regular sequence. 
But  in
the proof we present, we will see how $\Edot$
and $\Ecirc$ relate to
the volume of a subdivided simplex. 
The algorithms we present take
advantage of the polyhedral structure in our problem 
to give bounds (lower and upper-bound) for an excess number.
\begin{thm}
\label{thm:mainResult} Let $\fI$ be an ideal of $\mathbb{C}\left[x_{0},\dots,x_{n}\right]$
generated by $B_1,B_2,\dots,B_k$ such that $B_i=x_{i}^{p_{i}}$.
If $f_{1},\dots,f_{n}$
define a $\Bsys $  of degree $\dd$, then
\[
E_{\bullet}\left(\fI;d_{1},\dots,d_{n}\right)+p_{1}\cdots p_{k}\sum_{\delta=0}^{n-k}\left(\left(-1\right)^{\delta}{\cal D}_{n-k-\delta}{\cal P}_{\delta}\right)=d_{1}d_{2}\cdots d_{n}
\]
where ${\cal D}_{n-k-\delta}$ is the degree $n-k-\delta$ elementary symmetric function evaluated at  $d_{1},\dots,d_{n}$  and $P_{\delta}$
is the degree $\delta$ complete homogenous symmetric function evaluated at  $p_{1},\dots,p_{k}$. 
\end{thm}
The paper is structured as follows. We consider the case when $\fI$
is a monomial ideal, and show excess numbers equal mixed volumes of
polytopes (Lemma $\ref{thm:translateProblem}$). By further restricting
to the case when the ideal $\fI$ defines a complete intersection
that is also a linear space (though not necessarily reduced), we do
a mixed volume computation (Lemma $\ref{lem: finalIntegral}$) to
get an explicit formula for excess numbers. 
In the final section, we present our algorithms that take advantage of the first sections results.

\subsection*{Acknowledgements }

The author would like to thank Alicia Dickenstein for her many helpful comments
to improve this  paper as well as his advisor Bernd Sturmfels.

\section{The Monomial Case}

The key idea to Theorem $\ref{thm:mainResult}$ is to cast our excess
intersection problem in the language of combinatorial geometry and
prove Lemma $\ref{thm:translateProblem}$. Since the following proofs
will use Newton polytopes, Minkowski sums, and genericity, we set
up additional notation here. 

The $\emph{Newton polytope }$ of a form $f$ will be denoted as ${\cal N}\left(f\right)$.
The standard $n$-simplex is the convex hull of the origin $\epsilon_{0}$,
and the standard basis of unit vectors $\epsilon_{1},\dots,\epsilon_{n}$
in $\mathbb{R}^{n}$. The binary operation, Minkowski sum, will be
denoted as ``$+$''. 

Now, we give examples of $\Bsys$'s.

\begin{example}\label{trivalExample}
Let $\fI=\left(1\right)$ be the trivial ideal of $\mathbb{C}[x_0,x_1]$ so that $B_1=1$ generates $\fI$.
Then, a $\Bsys$ of degree $d_1$ is given by $1$ homogenous polynomials $f_1$
\[
f_{1}=a_{1}\cdot 1.
\]
Here, $a_{1}$ is a polynomial of degree $d_1$. 
For a general choice of $a_1$, the  excess number is $d_1$. When the coefficients of $a_1$ are specially chosen the excess number can decrease. In this case, the excess number can only be less than $d_i$ if the discriminant of $a_1$ vanishes. 
\end{example}

Whenever $\fI$ is a principal ideal, the excess numbers are
easy to determine algebraically. 
\begin{example}
Let $\fI=\left(B_1\right)$ be a principal ideal of the ring $\mathbb{C}\left[x_{0},\dots,x_{n}\right]$, and let $\deg B_1=p$.
Then, a $\Bsys$ of degree $\dd$ is given by $n$ homogenous polynomials $f_1,\dots,f_n:$
\[
f_{1}=a_{1}B_1,\, f_{2}=a_{2}B_1,\dots,f_{n}=a_{n}B_1.
\]
Here, $a_{i}$ is a polynomial of degree $d_i-p$. 
To determine the excess number $E_{\bullet}\left(\fI;d_{1},\dots,d_{n}\right)$,
we saturate $\left(f_{1},\dots,f_{n}\right)$ by $\fI$. Doing
so, we conclude that the excess intersection of $\left(f_{1},\dots,f_{n}\right)$
is defined by $\left(a_{1},\dots,a_{n}\right)$ and consists of finitely
many points. By Bezout's theorem, it follows that $E_{\bullet}\left(\fI;d_{1},\dots,d_{n}\right)\leq\left(d_{1}-p\right)\left(d_{2}-p\right)\cdots\left(d_{n}-p\right)$. 
When the $a_i$ are general polynomials of degree $d_i-p$, the equality above holds. 
\end{example}

\begin{example}\label{monTC}
Let the ideal $\fI\subset\mathbb{C}\left[x,y,z,w\right]$
be generated by the forms 
\[
B_{1}=z^{2},\: B_2=yw,\: B_{3}=yz,\: B_4=xw,\: B_{5}=y^{2},\:B_6=xz.
\]
Then, a $\Bsys$ of degree $(2,2,2)$ is a system of $3$ quadrics which are linear combinations of 
$B_1,\dots,B_6$.
A general $\Bsys$ in this case has four solutions not contained in $\textbf{V}(\fI)$. So $E_\bullet(\fI;2,2,2)=4$.
Four is also the the mixed volume of the Newton polytopes of $f_1,f_2,f_3$. In Lemma \ref{translateProblem}, we will see that this is not a coincidence. 
\end{example}

\begin{example}\label{TCFail}
Let the ideal $\fI\subset\mathbb{C}\left[x,y,z,w\right]$
be generated by the forms 
\[
B_{1}=z^{2}-yw,\: B_{2}=yz-xw,\: B_{3}=y^{2}-xz,
\]
Then a $\Bsys$ of degree $(2,2,2)$ is a system of $3$ quadrics which are linear combinations of 
$B_1,B_2,B_3$.
A general $\Bsys$ in this case is equal to the ideal $\fI$ and has no solutions outside of 
$\textbf{V}(\fI)$. So $E_\bullet(\fI;2,2,2)=0$.
\end{example}

In Example \ref{monTC} the excess number was a mixed volume of the Newton polytopes of $f_i$, but in  Example \ref{TCFail}, this was not the case.
Now, we explain the differences between these two situations. 

\begin{lem}\label{thm:translateProblem}
Let $\fI\subset\mathbb{C}\left[x_{0},\dots,x_{n}\right]$ be
an ideal defined by the monomials $B_1,\dots,B_l$. 
If $f_{1},\dots,f_{n}$ are a general  $\Bsys$ of degree $\dd$,  then the excess number $\Edot$
equals the mixed volume of the Newton polytopes ${\cal N}\left(f_{1}\right),\dots,{\cal N}\left(f_{n}\right)$. 
\end{lem}
\begin{proof}

The set of $\Bsys$'s of degree $\dd$ is parameterized by the coefficients of $a_{ij}$  in Definition \ref{defBsys}. 
We denote the projective space associated to the coefficients of all of the polynomials $a_{ij}$ by 
$\mathbb{P}_A$. The dimension of this projective space is one less than the number of coefficients of all the $a_{ij}$. 
We denote the projective space associated to the coefficients of the monomials of $f_1,f_2,\dots,f_n$ as 
$\mathbb{P}_C$. 
So we have a natural map $\psi$ from $\mathbb{P}_A$ to $\mathbb{P}_C$ that maps the coefficients of $a_{ij}$ to coefficients of $f_1,\dots,f_n$. 
Important for our situation, is that when $B_1,\dots,B_n$  are monomials, the image of this map $\psi$ is onto. 
Let $U_1$ denote a Zariski dense open subset of $\mathbb{P}_A$ whose complement contains coefficients that give rise to $\Bsys$'s with excess numbers less than expected. 
Let $V_2$ be a Zariski dense open set of coefficients of $f_1,\dots,f_n$ such that the mixed volume is less than expected. 
If $U_2$ be in the inverse image of $V_2$ under $\psi$,
then 
the intersection of $U_1$ and $U_2$ is again a Zariski dense open subset of $\mathbb{P}_A$. 
In particular, this means we can say a general $\Bsys$'s for which  $B_1,\dots,B_l$ are monomials has
$\Edot$ equal to the mixed volume of ${\cal N}\left(f_{1}\right),\dots,{\cal N}\left(f_{n}\right)$.
 \end{proof}

In Example \ref{monTC}, $\mathbb{P}_A$ equals $\mathbb{P}^9$ and $\mathbb{P}_C$ equals $\mathbb{P}^{17}$. By a dimension count we see that $\psi$ is not onto and explains why the mixed volume can be different from the excess number. 

The key idea of this proof was the notion of a "general $\Bsys$" 
so that we could use Bernstein's theorem to count  the solutions
we are interested in. So to determine excess numbers of monomial ideals,
we determine mixed volumes. In general, mixed volume computations
are complicated, but in some cases there is hope for an explicit formula.
The case we consider is when $\fI=\left(x_{1}^{p_{1}},\dots,x_{k}^{p_{k}}\right)$
is generated by powers of unknowns. If $f_{1},\dots,f_{n}$ define a $\Bsys$  of degree $\dd$, then we determine the excess
number $E_{\bullet}\left(\fI;d_{1},\dots,d_{n}\right)$ by calculating
the mixed volume of ${\cal N}\left(f_{1}\right),\dots,{\cal N}\left(f_{n}\right)$.
Recall that the mixed volume \cite{numBook} {[}Chapter 8.5{]} can
be calculated by determining the coefficient of $\lambda_{1}\lambda_{2}\cdots\lambda_{n}$
in the polynomial defining the volume of the scaled Minkowski sum
\[
\lambda_{1}{\cal N}\left(f_{1}\right)+\cdots+\lambda_{n}{\cal N}\left(f_{n}\right).
\]
If we let $\mathfrak{S}$ denote a simplex, then Lemma $\ref{lem: simplexBounds}$
shows that slicing $\mathfrak{S}$ by an appropriate hyperplane subdivides
$\mathfrak{S}$ into two convex polytopes $\mathfrak{S}_{0}$ and
$\mathfrak{S}_{1}$. The hyperplane can be chosen so that $\mathfrak{S}_{1}=\lambda_{1}{\cal N}\left(f_{1}\right)+\cdots+\lambda_{n}{\cal N}\left(f_{n}\right)$.
Since $\Vol\mathfrak{S}=\Vol\mathfrak{S}_{0}+\Vol\mathfrak{S}_{1}$,
we compute our desired mixed volume by determining the coefficients
of $\lambda_{1}\cdots\lambda_{n}$ in $\Vol\mathfrak{S}_{0}$ and
$\Vol\mathfrak{S}$ (Lemma $\ref{lem: finalIntegral}$). 

To elucidate our ideas we consider the following example.
\begin{example}
Let $\fI=\left(x_{1}^{p_{1}},x_{2}^{p_{2}}\right)$ be an ideal
of the ring $\mathbb{C}\left[x_{0},x_{1},x_{2},x_{3}\right]$ and
let $f_{1},f_{2},f_{3}$ be a $\Bsys$ of degree $(d_{1},d_{2},d_{3})$.
Then the Newton polytope of $f_{i}$ will be the convex hull of two
tetrahedra. Specifically, the Newton polytope of $f_{i}$ is the convex
hull of the following eight points in $\mathbb{R}^{3}$ of which six
are vertices of ${\cal N}\left(f_{i}\right)$: 

\[
\begin{array}{cccc}
\left(p_{1},0,0\right) & \left(d_{i},0,0\right) & \left(p_{1},d_{i}-p_{1},0\right) & \left(p_{1},0,d_{i}-p_{1}\right)\\
\left(0,p_{2},0\right) & \left(0,d_{i},0\right) & \left(p_{2},d_{i}-p_{2},0\right) & \left(p_{2},0,d_{i}-p_{2}\right).
\end{array}
\]
 To avoid confusion with points in $\mathbb{P}^{3}$, we describe
points in $\mathbb{R}^{3}$ as $\left(u_{1},u_{2},u_{3}\right)$ rather
than $\left(x_{1},x_{2},x_{3}\right)$. 
We will also describe the Newton
polytope ${\cal N}\left(f_{i}\right)$ by its $5$ supporting hyperplanes rather than its vertices:

$\bullet$ the $3$ coordinate hyperplanes,

$\bullet$ the hyperplane defined by $u_{1}+u_{2}+u_{3}-d_{i}$, and 

$\bullet$ the hyperplane defined by $\frac{u_{1}}{p_{1}}+\frac{u_{2}}{p_{2}}-1$.
\\
The normal vectors of the $5$ hyperplanes supporting ${\cal N}\left(f_{i}\right)$
are the same for every $i$. Indeed they are the standard unit vectors
$\epsilon_{1},\epsilon_{2},\epsilon_{3}$, the vector $\epsilon_{1}+\epsilon_{2}+\epsilon_{3}$,
and the vector $\frac{1}{p_{1}}\epsilon_{1}+\frac{1}{p_{2}}\epsilon_{2}$.
By standard polytope theory \cite{polytopeBook}{[}Proposition 7.12{]},
it follows that the scaled Minkowski sum $\lambda_{1}{\cal N}\left(f_{1}\right)+\lambda_{2}{\cal N}\left(f_{2}\right)+\lambda_{3}{\cal N}\left(f_{3}\right)$,
has the same $5$ normal vectors as those of ${\cal N}\left(f_{i}\right)$.
Indeed, the supporting hyperplanes are 

$\bullet$ the $3$ coordinate hyperplanes,

$\bullet$ the hyperplane defined by $u_{1}+u_{2}+u_{3}-\left(\lambda_{1}d_{1}+\lambda_{2}d_{2}+\lambda_{3}d_{3}\right)$,
and 

$\bullet$ the hyperplane defined by $\frac{u_{1}}{p_{1}}+\frac{u_{2}}{p_{2}}-\left(\lambda_{1}+\lambda_{2}\right)$.\\
Now note that four of the five hyperplanes are defining facets of
a simplex whose volume is $\frac{1}{3!}\left(\lambda_{1}d_{1}+\lambda_{2}d_{2}+\lambda_{3}d_{3}\right)^{3}$.
The fifth hyperplane subdivides the simplex as seen in Figure $1$.
By subtracting the volume of the white figure from the volume of the
simplex, we attain the mixed volume by considering the coefficients
of $\lambda_{1}\lambda_{2}\lambda_{3}$ in the difference. In this
example, we would find the excess number satisfies $E_{\bullet}\left(x_{1}^{p_{1}},x_{2}^{p_{2}};d_{1},d_{2},d_{3}\right)+p_{1}p_{2}\left(d_{1}+d_{2}+d_{3}-p_{1}-p_{2}\right)=d_{1}d_{2}d_{3}$.
\end{example}
\begin{figure}[h]
\caption{\protect\includegraphics[scale=0.24]{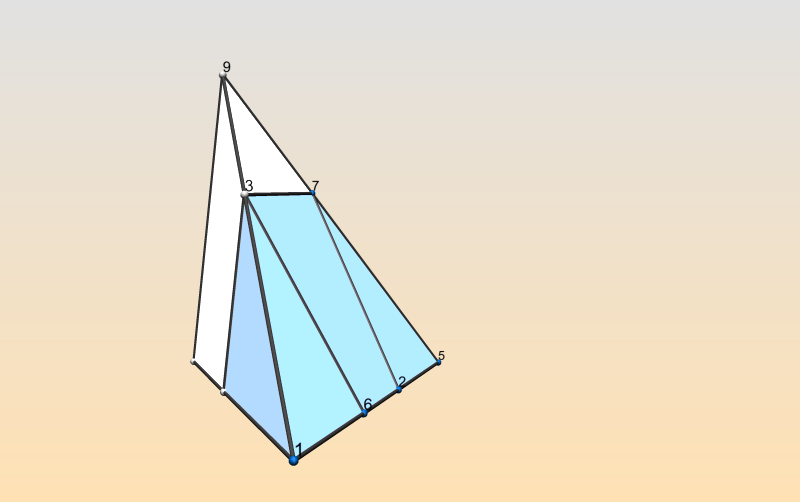}}

\end{figure}

We now precisely define the polytopes $\mathfrak{S},\mathfrak{S}_{0},\mathfrak{S}_{1}$.
Let $D=\lambda_{1}d_{1}+\cdots+\lambda_{n}d_{n}$ and $\Lambda=\lambda_{1}+\cdots+\lambda_{n}$.
Then $\mathfrak{S}$ is defined as the $n$-simplex whose $n+1$  vertices
are the origin and $D\epsilon_{i}$. Moreover, the volume of $\mathfrak{S}$
equals $D^{n}/n!$ which has a term $d_{1}d_{2}\cdots d_{n}\lambda_{1}\lambda_{2}\cdots\lambda_{n}$
when expanded out. If we define the hyperplane $h_{p}$ by $\frac{u_{1}}{p_{1}}+\cdots+\frac{u_{k}}{p_{k}}-\Lambda$,
then $h_{p}$ slices $\mathfrak{S}$ into two convex polytopes $\mathfrak{S}_{0}$
containing the origin, and $\mathfrak{S}_{1}$. We will prove that
$\mathfrak{S}_{1}$ is the scaled Minkowski sum $\lambda_{1}{\cal N}\left(f_{1}\right)+\cdots+\lambda_{n}{\cal N}\left(f_{n}\right)$.

\begin{lem}
\label{lem: simplexBounds} Let $f_{1},\dots,f_{n}$ be a $\Bsys$
of degree $\dd$. If $\fI=\left(x_{1}^{p_{1}},\dots,x_{k}^{p_{k}}\right)$,
then $\mathfrak{S}_{1}=\lambda_{1}{\cal N}\left(f_{1}\right)+\cdots+\lambda_{n}{\cal N}\left(f_{n}\right)$. \end{lem}
\begin{proof}
We will show that the polytopes $\mathfrak{S}_{1}$ and $\lambda_{1}{\cal N}\left(f_{1}\right)+\cdots+\lambda_{n}{\cal N}\left(f_{n}\right)$
have the same $n+2$  supporting hyperplanes so that they must be
equal. Specifically, we show that the supporting hyperplanes are 

$\bullet$ the $n$ coordinate hyperplanes,

$\bullet$ the hyperplane $h_{d}$ defined by $u_{1}+\cdots+u_{n}-D$,
and 

$\bullet$ the hyperplane $h_{p}$ defined by $\frac{u_{1}}{p_{1}}+\cdots+\frac{u_{k}}{p_{k}}-\Lambda$.\\
The first $n+1$  hyperplanes support the simplex $\mathfrak{S}$.
The last hyperplane $h_{p}$ slices $\mathfrak{S}$ into two polytopes
one containing the origin and a second which is by definition the
polytope $\mathfrak{S}_{1}$. Now because $\lambda_{i}{\cal N}\left(f_{i}\right)$
has $n+2$ supporting hyperplanes consisting of 

$\bullet$ the $n$ coordinate hyperplanes,

$\bullet$ the hyperplane defined by $u_{1}+\cdots+u_{n}-\lambda_{i}d_{i}$,
and 

$\bullet$ the hyperplane defined by $\frac{u_{1}}{p_{1}}+\cdots+\frac{u_{k}}{p_{k}}-\lambda_{i}$
.\\
By standard polytope theory \cite{polytopeBook}{[}Proposition 7.12{]},
it follows $\lambda_{1}{\cal N}\left(f_{1}\right)+\cdots+\lambda_{n}{\cal N}\left(f_{n}\right)$
has the desired supporting hyperplanes. 
\end{proof}
With Lemma $\ref{lem: simplexBounds}$, we are able to calculate the
mixed volume by determining the coefficient of $\lambda_{1}\lambda_{2}\cdots\lambda_{n}$
in an integral as seen in Lemma $\ref{lem: finalIntegral}$. 
\begin{lem}
\label{lem: finalIntegral}With the previous notation, $\Vol\left(\mathfrak{S}_{0}\right)$
is a polynomial whose coefficient of $\lambda_{1}\lambda_{2}\cdots\lambda_{n}$
equals \textup{$p_{1}\cdots p_{k}\sum_{\delta=0}^{n-k}\left(\left(-1\right)^{\delta}{\cal D}_{n-k-\delta}\cdot{\cal P}_{\delta}\right)$. }\end{lem}
\begin{proof}
By Lemma $\ref{lem: simplexBounds}$, it follows that the volume of
$\mathfrak{S}_{0}$ equals 

\[
\Vol\left(\mathfrak{S}_{0}\right)=\int\left[\int\cdots\int d_{x_{n}}d_{x_{n-1}}\cdots d_{x_{k+1}}\right]d_{\triangle}
\]
with the bounds of each integral inside the brackets with respect
to $d_{x_{i}}$ being $\left(x_{i}=0\right)\to\left(x_{i}=D-x_{i-1}-\cdots-x_{1}\right)$
and $\triangle$ denotes the simplex in $k$-dimensional space with
$k+1$ vertices of $\epsilon_{0},\Lambda p_{1}\cdot\epsilon_{1},\dots,\Lambda p_{k}\cdot\epsilon_{k}$. 

By using the calculus fact $\left[\int\cdots\int d_{x_{n}}d_{x_{n-1}}\cdots d_{x_{k+1}}\right]=\frac{1}{r!}\left(D-x_{k}-\cdots-x_{1}\right)^{r}$
and the binomial theorem, we have

\[
\begin{array}{lll}
\Vol\left(\mathfrak{S}_{0}\right) & = & \frac{1}{r!}\int\left(D-x_{k}-\cdots-x_{1}\right)^{r}d_{\triangle}\\
 & = & \frac{1}{r!}\sum_{\delta=0}^{r}\left(\left(-1\right)^{\delta}\binom{r}{\delta}D^{r-\delta}\int\left(x_{k}+\cdots+x_{1}\right)^{\delta}d_{\triangle}\right)\\
 & = & p_{1}\cdots p_{k}\sum_{\delta=0}^{r}\left(\left(-1\right)^{\delta}\frac{D^{n-k-\delta}}{\left(n-k-\delta\right)!}\frac{\Lambda^{k+\delta}}{\left(\delta+k\right)!}{\cal P}_{\delta}\right)
\end{array}
\]
with $r=n-k$. It is known how to integrate a linear form raised to
some power over the simplex. So to get the last equality, we use \cite{integrateSimplex}
{[}Remark $9${]}, that says $\int\left(x_{k}+\cdots +x_{1}\right)^{\delta}d_{\triangle}=\Lambda^{k+\delta}p_{1}\cdots p_{k}\frac{\delta!}{\left(\delta+k\right)!}{\cal P}_{\delta}$.%

Now, note that $L_{m}\left(\lambda_{1},\dots,\lambda_{n}\right):=m!\left(\text{the monomials in \ensuremath{\lambda_{1},\dots,\lambda_{n}}\,of degree }m\right)$
is congruent to $\Lambda^{m}$ modulo $\lambda_{1}^{2},\dots,\lambda_{n}^{2}$.
Similarly, also note that $D^{m}$ is congruent to $L_{m}\left(\lambda_{1}d_{1},\dots,\lambda_{n}d_{n}\right)$
modulo $\lambda_{1}^{2},\dots,\lambda_{n}^{2}$. So we have 
\[
\begin{array}{lll}
\Vol\left(\mathfrak{S}_{0}\right) & \equiv & p_{1}\cdots p_{k}\sum_{\delta=0}^{n-k}\left(\left(-1\right)^{\delta}L_{n-k-\delta}\left(\lambda_{1}d_{1},\dots,\lambda_{n}d_{n}\right)\cdot L_{k+\delta}\left(\lambda_{1},\dots,\lambda_{n}\right){\cal P}_{\delta}\right)\\
 & \equiv & \lambda_{1}\cdots\lambda_{n}\cdot p_{1}\cdots p_{k}\sum_{\delta=0}^{n-k}\left(\left(-1\right)^{\delta}{\cal D}_{n-k-\delta}{\cal P}_{\delta}\right).
\end{array}
\]
The last congruence is shown by an easy combinatorial argument. 
\end{proof}
\begin{rem}
We remark that the number $E_{\bullet}\left(\fI;d_{1},\dots,d_{n}\right)$
depends only on the Newton polytopes ${\cal N}\left(f_{1}\right),\dots,{\cal N}\left(f_{n}\right)$.
For example, consider the ideals $\fI_{1}=\left(x^{3},y^{3}\right)$,
$\fI_{2}=\left(x^{3},y^{3},x^{2}y,xy^{2}\right)$, $\fI_{3}=\left(x^{3},y^{3},x^{2}y^{2}\right)$
in the ring $\mathbb{C}\left[w,x,y,z\right]$. All three of these
ideals have the same excess numbers when every $d_{i}$ is greater
than $4$ because the Newton polytopes of the defining polynomials of a $\Bsys$ 
 are the same for $i=1,2,3$. In particular, $E_{\bullet}\left(\fI_{i};5,5,5\right)=44$
for $i=1,2,3$. But if we consider the ideal ${\cal J}=\left(x^{3},y^{3}\right)+\left(xy\right)$,
we find the Newton polytopes of a $\BsysJ$  are different
from those Newton polytopes of  a $\Bsys$. In
particular, one can compute the excess number $E_{\bullet}\left({\cal J};5,5,5\right)$
to be $65$.
\end{rem}

\section{Numerical Algebraic Geometry Algorithms }

We have given a combinatorial description of excess numbers of monomial
ideals in the first part of the paper and used this idea to give an
explicit formula in Theorem $\ref{thm:mainResult}$. In the last part
of this paper, we give algorithms that use homotopy continuation,
an idea from numerical algebraic geometry, to compute excess numbers
of any ideal $\fI\subset\mathbb{C}\left[x_{0},\dots,x_{n}\right]$. 
As mentioned in the introduction there are other ways to compute excess numbers with Segre classes. In addition, one can use off-the-shelf computer algebra software like \texttt{Macaulay2}
to compute excess numbers by saturating the ideal of a $\Bsys$ by $\fI$. 
Also, the examples we present here can also be worked out by hand using Fulton-MacPherson intersection theory.


Our algorithms will construct two homotopies, called $\hup$ and $\hit$,
that take the isolated solutions of a $\BsysPrime$
$f_{1}',\dots,f_{n}'$ as start points and tracks them to solutions
of $f_{1},\dots,f_{n}$ giving bounds on $E_{\bullet}\left(\fI;d_{1},\dots,d_{n}\right)$.
In the first algorithm, the monomial ideal $\fI'$ is constructed
so that $E_{\bullet}\left(\fI';d_{1},\dots,d_{n}\right)\geq E_{\bullet}\left(\fI;d_{1},\dots,d_{n}\right)$.
By doing a numerical membership test \cite{numBook} {[}Chapter $15${]},
we will determine $E_{\bullet}\left(\fI;d_{1},\dots,d_{n}\right)$
and isolated solutions of ${\bf V}\left(f_{1},\dots,f_{n}\right)\backslash{\bf V}\left(\fI\right)$
explicitly. In the second algorithm, the monomial ideal $\fI'$
is constructed to give lower bounds of $E_{\bullet}\left(\fI;d_{1},\dots,d_{n}\right)$
instead. But by iterating the second algorithm, we have a probabilistic
way to make this bound sharp  and compute all isolated solutions of
${\bf V}\left(f_{1},\dots,f_{n}\right)\backslash{\bf V}\left(\fI\right)$
explicitly. The $\hup$-homotopy gets its name because it produces
an $upp$er bound of $E_{\bullet}\left(\fI;d_{1},\dots,d_{n}\right)$
prior to a membership test. The $\hit$-homotopy gets its name because
several $ite$rations can produce sharp lower bounds of $E_{\bullet}\left(\fI;d_{1},\dots,d_{n}\right)$
after a membership test.

\subsection{Algorithm one and the $\hup$-homotopy }

We now give a definition of the $\hup$-homotopy and prove that it
does indeed provide an upper bound of $E_{\bullet}\left(\fI;d_{1},\dots,d_{n}\right)$
prior to a membership test. 
\begin{defn}
Let $B_{1},\dots B_{l}\in\mathbb{C}\left[x_{0},\dots,x_{n}\right]$
be forms such that $B_{j}=\sum_{k}A_{j,k}$ with $A_{j,k}$ being
a monomial multiplied by a scalar. To ease notation, let $\overrightarrow{A_{j}}=\left[A_{j,1},\dots,A_{j,k_{j}}\right]$
be a row vector whose entries sum to $B_j$, $\overrightarrow{\alpha_{i,j}}$ be a row vector
of $k_{j}$ different general forms, and $\overrightarrow{\beta_{i,j}}$
be a row vector of a general form that is repeated $k_{j}$ times. Define
the $\hup$-homotopy as $\hup\left(t;d_{1},\dots,d_{n}\right):=$
\[
\left(t\left[\begin{array}{cccc}
\overrightarrow{\alpha_{1,1}} & \overrightarrow{\alpha_{1,2}} & \cdots & \overrightarrow{\alpha_{1,l}}\\
\vdots &  &  & \vdots\\
\overrightarrow{\alpha_{n,1}} & \overrightarrow{\alpha_{n,2}} & \cdots & \overrightarrow{\alpha_{n,l}}
\end{array}\right]+\left(1-t\right)\left[\begin{array}{cccc}
\overrightarrow{\beta_{1,1}} & \overrightarrow{\beta_{1,2}} & \cdots & \overrightarrow{\beta_{1,l}}\\
\vdots &  &  & \vdots\\
\overrightarrow{\beta_{n,1}} & \overrightarrow{\beta_{n,2}} & \cdots & \overrightarrow{\beta_{n,l}}
\end{array}\right]\right)\left[\begin{array}{c}
\overrightarrow{A_{1}}\mbox{}^{T}\\
\overrightarrow{A_{2}}\mbox{}^{T}\\
\vdots\\
\overrightarrow{A_{l}}\mbox{}^{T}
\end{array}\right],
\]
with the degrees of the general forms of $\overrightarrow{\alpha_{i,j}}$
and $\overrightarrow{\beta_{i,j}}$ chosen so that $\hup\left(t,d_{1},\dots,d_{n}\right)$
is a system of $n$ forms of degrees $d_{1},\dots,d_{n}$. We denote the start points of $\hup$$\left(t;d_{1},\dots,d_{n}\right)$
as $S_{\hup}$ and take them to be the isolated solutions of $ $$\hup\left(1;d_{1},\dots,d_{n}\right)$.
Denote the end points of $\hup$$\left(t;d_{1},\dots,d_{n}\right)$
as $T_{\hup}$.

With this definition, we have when $t=1$ that $\hup\left(t;d_{1},\dots,d_{n}\right)$
is a general $\BsysPrime$ $f_{1}',\dots,f_{n}'$
of degrees $\dd$. On the other hand, when $t=0$ we
have $\hup\left(t;d_{1},\dots,d_{n}\right)$ is a $\Bsys$ of degree $\dd$. 
By the fundamental theorem
of parameter continuation of isolated roots \cite{numBook} {[}Theorem
$7.1.6${]} it follows that $T_{\hup}$ contains all isolated solutions
of $f_{1},\dots,f_{n}$. In particular, this proves Theorem $\ref{homotopyUpBound}$
because $\#S_{\hup}\geq\#T_{\hup}$. \end{defn}
\begin{thm}
\label{homotopyUpBound}Let $\fI\subset\mathbb{C}\left[x_{0},\dots,x_{n}\right]$
be generated by the forms $B_{1},\dots B_{l}$ such that $B_{j}=\sum_{k}A_{j,k}$
with $A_{j,k}$ being a monomial multiplied by a scalar. If we let $\fI'$ be generated
by $A_{j,k}$, then 

\[
E_{\bullet}\left(\fI';d_{1},\dots,d_{n}\right)\geq E_{\bullet}\left(\fI;d_{1},\dots,d_{n}\right).
\]
Moreover, the parameter homotopy $\hup\left(t;d_{1},\dots,d_{n}\right)$
has endpoints $T_{{\bf \hup}}$ containing all isolated solutions
of $f_{1}',\dots,f_{n}'$. 
\end{thm}
Now that we have the theorem, we present our algorithm.
\\ 

\framebox{\begin{minipage}[t]{.9\columnwidth}%
${\bf Input}$: Natural numbers $d_{1},\dots,d_{n}$ and generators
$B_{1},\dots,B_{l}$ of an ideal $\fI$ in $\mathbb{C}\left[x_{0},\dots,x_{n}\right]$
such that $B_{j}=\sum_{k}A_{j,k}$. 

${\bf Output:}$ The excess number $E_{\bullet}\left(\fI;d_{1},\dots,d_{n}\right)$. 

$\emph{Step 1}:$ Construct the the $\hup$-homotopy ${\bf \hup}\left(t;d_{1},\dots,d_{n}\right)$.

$\emph{Step 2:}$ Solve the start system $\hup\left(1;d_{1},\dots,d_{n}\right):=\left[f_{1}',\dots,f_{n}'\right]^{T}$
and compute $S_{\hup}$.

$\emph{Step 3:}$ Use the $\hup$-homotopy to determine $T_{\hup}$
and an upper bound of $E_{\bullet}\left(\fI;d_{1},\dots,d_{n}\right)$. 

$\emph{Step 4:}$ Use a numerical membership test to determine $E_{\bullet}\left(\fI;d_{1},\dots,d_{n}\right)$
and the isolated solutions of $\hup\left(0;d_{1},\dots,d_{n}\right):=f_{1},\dots,f_{n}$. %
\end{minipage}}
\\ 

For this algorithm, we assume in Step $2$ that the excess intersection
of a monomial ideal can be determined.  We now give an example where $\fI$ defines the twisted
cubic. 
\begin{example}
\label{exAlg1} Let the ideal $\fI\subset\mathbb{C}\left[x,y,z,w\right]$
be generated by the forms 
\[
B_{1}=z^{2}-yw,\: B_{2}=yz-xw,\: B_{3}=y^{2}-xz,
\]
and suppose we want to calculate $E_{\bullet}\left(\fI;3,3,3\right)$.
To run the first algorithm, we input $d_{1}=d_{2}=d_{3}=3$ and $B_{1},B_{2},B_{3}$.
In Step $1$, we determine $\hup\left(t;d_{1},\dots,d_{n}\right)=$
\[
\left(t\left[\begin{array}{cccc}
a{}_{11} & a{}_{12} & \cdots & a{}_{16}\\
a{}_{21} & a{}_{22} & \cdots & a{}_{26}\\
a{}_{31} & a{}_{32} & \cdots & a{}_{36}
\end{array}\right]+\left(1-t\right)\left[\begin{array}{cccccc}
b_{11} & b_{11} & b_{12} & b_{12} & b_{13} & b_{13}\\
b_{21} & b_{21} & b_{22} & b_{22} & b_{23} & b_{23}\\
b_{31} & b_{31} & b_{32} & b_{32} & b_{33} & b_{33}
\end{array}\right]\right)\left[\begin{array}{c}
z^{2}\\
-yw\\
yz\\
-xw\\
y^{2}\\
-xz
\end{array}\right].
\]
 The forms $a{}_{ij}$ and $b_{ij}$ are general linear forms of $\mathbb{C}\left[x,y,z,w\right]$.
Once we have solved the system $\hup\left(1;3,3,3\right)$ in Step
$2$, we path track in Step $3$ to calculate $T_{\hup}$ giving an
upper bound $\#T_{\hup}$ of $E_{\bullet}\left(\fI;3,3,3\right)$.
In Step $4$, we use a numerical membership test \cite{numBook} to
conclude $E_{\bullet}\left(\fI;3,3,3\right)=10$. Indeed, if 

\[
\left[\begin{array}{c}
b_{11}\\
b_{12}\\
\vdots\\
b_{33}
\end{array}\right]=\left[\begin{array}{cccccccc}
1/2 & 1 & 4/5 & 1/3 & 1/5 & 7/8 & 13 & 1/3\\
3 & 7 & 9/7 & 1/8 & 4 & 1/6 & 5 & -1\\
-5 & 4 & 7/8 & 8/9 & 3 & 1/15 & 1/6 & -8\\
-1/4 & 2 & 1/3 & -1 & -1 & -2 & 7/9 & 1/4
\end{array}\right]^{T}\left[\begin{array}{c}
x\\
y\\
z\\
w
\end{array}\right],
\]
$ $then we find the ten excess points are $s_{1},s_{2},s_{3},s_{4},s_{5},s_{6},s_{7},s_{8}=\bar{s}_{5},s_{9}=\bar{s}_{6},s_{10}=\bar{s}_{7}$:

\[
\begin{array}{ccccc}
 & s_{1} & s_{2} & s_{3} & s_{4}\\
x & -6.1999 & -0.2081 & -1.0024 & -0.1530\\
y & 5.9766 & 0.5979 & 3.1208 & 0.3771\\
z & -2.3702 & -2.1386 & -5.1077 & -0.6183\\
w & 1 & 1 & 1 & 1
\end{array},
\]
 
\[
\begin{array}{cccc}
 & s_{5} & s_{6} & s_{7}\\
x & -.6493+1.4057i & 0.4713-0.0461i & 2.9076+0.0384i\\
y & .4134-1.4061i & 0.2603-0.5271i & -1.0341+1.7553i\\
z & -1.1267+0.3173i & -0.9278+0.1923i & -0.7082-1.2392i\\
w & 1 & 1 & 1
\end{array}.
\]
 
\end{example}

\subsection{Algorithm two and the $\hit$-homotopy}

The second algorithm we present is probabilistic. The algorithm uses the $\hit$-homotopy
to compute a lower bound of excess numbers. By iterating this algorithm,
the lower bounds can become sharp. 
\begin{defn}
Let $B_{1},\dots B_{l}\in\mathbb{C}\left[x_{0},\dots,x_{n}\right]$
be forms that generate $\fI$ and $A_{1},\dots,A_{l}$ be monomials
that generate $\fI'$ such that $\deg B_{j}=\deg A_{j}$. The
$\hit$-homotopy is defined as $\hit\left(t;d_{1},\dots,d_{n}\right):=$
\[
\left[\begin{array}{cccc}
a_{11} & a_{12} & \cdots & a_{1l}\\
\vdots\\
a_{n1} & a_{n2} & \cdots & a_{nl}
\end{array}\right]\left(t\left[\begin{array}{c}
A_{1}\\
\vdots\\
A_{l}
\end{array}\right]+\gamma\left(1-t\right)\left[\begin{array}{c}
B_{1}\\
\vdots\\
B_{l}
\end{array}\right]\right),
\]
 with the degrees of the general forms $a_{ij}$ equal to $\deg f_{i}-\deg A_{j}$.
We denote the start points of $\hit$ as $S_{\itt}$ and take them
to be the isolated solutions of $ $$\hit\left(1;d_{1},\dots,d_{n}\right)$
and denote the end points of $\hit\left(t;d_{1},\dots,d_{n}\right)$
as $T_{\itt}$.

With this definition, we have when $t$ equals $1$ that \break$\hit\left(t;d_{1},\dots,d_{n}\right)$
is a $\BsysPrime$ $f_{1}',\dots,f_{n}'$
of degrees $\dd$. On the other hand, when $t=0$ we
have $\hit\left(t;d_{1},\dots,d_{n}\right)$ is a system of $n$ $\fI$-general
forms of degrees $d_{1},\dots,d_{n}$. While the $\hit$-homotopy
is easy to set up,  the fundamental theorem of parameter continuation
of isolated roots \cite{numBook} {[}Theorem $7.1.6${]} 
cannot be applied. 
So $T_{\hit}$
does not necessarily contain all isolated solutions of $f_{1},\dots,f_{n}$.
However, after doing a membership test, we can determine some points
in $T_{\hit}$ are isolated solutions of $f_{1},\dots,f_{n}$. So
what we have is a lower bound on $E_{\bullet}\left(\fI;d_{1},\dots,d_{n}\right)$.
But, by iterating this homotopy, we can find more isolated solutions
and give a better lower bound. 
\end{defn}
\framebox{\begin{minipage}[t]{.9\columnwidth}%
${\bf Input}$: Natural numbers $d_{1},\dots,d_{n}$, generators $B_{1},\dots,B_{l}$
of an ideal $\fI\subset\mathbb{C}\left[x_{0},\dots,x_{n}\right]$,
monomials $A_{1},A_{2},\dots,A_{l}$ such that $\deg A_{j}=\deg B_{j}$,
and a (possibly empty) set $W$ of isolated solutions of $f_{1},\dots,f_{n}$. 

${\bf Output:}$ A set $W_{\itt}$ containing $W$ of isolated solutions
of $f_{1},\dots,f_{n}$, and $\#W_{\itt}$ a lower bound for the excess
number $E_{\bullet}\left(\fI;d_{1},\dots,d_{n}\right)$. 

$\emph{Step 1}:$ Construct the $\hit$-homotopy $\hit\left(t;d_{1},\dots,d_{n}\right)$
and track start solutions $S_{\itt}$ to target solutions $T_{\itt}$. 

$\emph{Step 2:}$ Use a membership test to determine which solutions
of $T_{\itt}$ are isolated and set $W_{\itt}$ to be the union of
$W$ and isolated solutions of $T_{\itt}$. 

$\emph{Step 3:}$ Output $W_{\itt}$ and $\#W_{\itt}$ OR repeat steps
$1-3$ by making a different choice of $\gamma$ in the $\hit$-homotopy.%
\end{minipage}}
\\ 

By taking different choices of $\gamma$ in the $\hit$-homotopy
we were able to produce the following example.
\begin{example}
If we take $A_{1}=z^{2},A_{2}=yz,A_{3}=y^{2}$ , then we have the
excess number $E_{\bullet}\left(\fI';3,3,3\right)=7$. Next,
we construct the $\hit$-homotopy as $\hit\left(t;d_{1},\dots,d_{n}\right)=$

\[
\left[\begin{array}{ccc}
b_{11} & b_{12} & b_{13}\\
b_{21} & b_{22} & b_{23}\\
b_{33} & b_{32} & b_{33}
\end{array}\right]\left(t\left[\begin{array}{c}
z^{2}\\
yz\\
y^{2}
\end{array}\right]+\gamma\left(1-t\right)\left[\begin{array}{c}
z^{2}-yw\\
yz-xw\\
y^{2}-xz
\end{array}\right]\right)
\]
with $a_{ij}$ the same as in Example $\ref{exAlg1}$. We find the
$7$ isolated solutions of $\hit\left(1;d_{1},d_{2},d_{3}\right)$
are $s_{1}',s_{2}',s_{3}',s_{4}',\bar{s}_{2}',\bar{s}_{3}',\bar{s}_{4}'$: 

\[
\begin{array}{ccccc}
 & s'_{1} & s'_{2} & s'_{3} & s'_{4}\\
x & -8.4814 & -.0354+.7868i & .8876+.0702i & -.3053+.4774i\\
y & 8.2976 & -.1201+-.7446i & .3006-.5880i & .3779-.7007i\\
z & -2.9043 & -.9638+.1650i & -1.2929+.2635i & -1.8276+.6092i\\
w & 1 & 1 & 1 & 1
\end{array}.
\]
By taking $\gamma$ to be different  complex numbers and keeping
$b_{ij}$ fixed, with $4$ iterations, we were able to find that $E_{\bullet}\left(\fI;d_{1},\dots,d_{n}\right)$
has a lower bound of $10$. By the previous subsection, we know that
this lower bound is actually sharp.
\end{example}
We comment that with some choices of $A_{1},A_{2},A_{3}$ defining
$\fI$, it can happen that $E_{\bullet}\left(\fI';3,3,3\right)$
is greater than, equal to, or less than $E_{\bullet}\left(\fI;3,3,3\right)$.
So one may be tracking too many paths, too few paths, or perhaps luckily
the right number. Open questions remain about for which choice of
monomials $A_{1},\dots,A_{l}$ yield the best computational results.
In addition, how should we choose $\gamma$ to find new solutions
as we iterate; and how can we verify that our lower
bound has become sharp are also interesting questions. These questions
will remain for future work, and their answers may depend heavily
on the context of the problem. 
\begin{rem}
We remark that the $\hit$-homotopy need not have had the $A_{j}$
be monomials. Any choice of a form $A_{j}$ whose degree equals $B_{j}$
could have been used. However, in this section, we have made the assumption
that excess intersections of monomial ideals can be computed effectively,
as we saw combinatorics can be used to understand excess numbers of
monomial ideals. 
\end{rem}
To conclude, we have shown that determining excess numbers of monomial
ideals can be reduced to computing a mixed volume in some cases.
With this idea, we are able to provide an explicit formula for excess
numbers of ideals with general generators. We presented two algorithms
using numerical algebraic geometry to determine excess numbers of
any ideal. We also demonstrated that these algorithms have successfully
lead to the calculation of excess numbers of an ideal defining the
twisted cubic.
 We believe that the the $\hup$-homotopy can compute excess numbers of many other ideals defined by sparse forms in many unknowns.

\bibliographystyle{plain}
\nocite{*}

(Jose Israel Rodriguez) \texttt{Department of Mathematics, The University of\break
California at
 Berkeley, 970 Evans Hall 3840, Berkeley, CA 94720-3840
USA }

$\emph{E-mail address}$: jo.ro@berkeley.edu 
\end{document}